\documentclass[12pt,reqno]{amsart}
\usepackage{tikz}
\usetikzlibrary{positioning}
\usetikzlibrary{arrows}
\usetikzlibrary{shapes.multipart}

\usepackage[T1]{fontenc}
\usepackage[utf8]{inputenc}
\usepackage{appendix}
\usepackage{paralist}
\usepackage{longtable}
\usepackage{graphicx}
\usepackage{amsmath}
\usepackage{fullwidth}
\usepackage{amssymb}
\usepackage{upgreek}
\usepackage{marvosym}
\usepackage{amsthm}
\usepackage{caption}
\usepackage{subcaption}
\usepackage{amstext}
\usepackage{array}
\usepackage{euler}
\usepackage{times}
\usepackage{lipsum}
\usepackage{xspace}

\usepackage{tabularx,ragged2e,booktabs,caption}
\newcolumntype{C}[1]{>{\Centering}m{#1}}

\usepackage{accents}

\usepackage{kpfonts}
\usepackage[margin=1in]{geometry}
\numberwithin{equation}{section}
\usepackage{algorithm}
\usepackage{algorithmic}
\makeatletter
\def\listofalgorithms{\@starttoc{loa}\listalgorithmname}
\def\l@algorithm{\@tocline{0}{3pt plus2pt}{0pt}{1.9em}{}}
\renewcommand{\ALG@name}{Algorithm}
\renewcommand{\listalgorithmname}{List of \ALG@name s}
\numberwithin{algorithm}{section}
\theoremstyle{definition}
\newtheorem{defn}[algorithm]{Definition}
\theoremstyle{remark}
\newtheorem{rem}[algorithm]{Remark}
\theoremstyle{theorem}
\newtheorem{thm}[algorithm]{Theorem}
\theoremstyle{proposition}
\newtheorem{prop}[algorithm]{Proposition}
\theoremstyle{example}

\newtheorem{exmp}{Example}[section]
\theoremstyle{corollary}

\makeatother

\usepackage{scalerel}

\usepackage{hyperref}
\hypersetup{
    colorlinks=true,
    linkcolor=blue,
    filecolor=magenta,
    urlcolor=cyan,
}
\begin{document}
\title[Value function reformulation for multiobjective bilevel optimization]{\Large{E\MakeLowercase{xtension of the value function reformulation}\\[0.8ex] \MakeLowercase{to multiobjective bilevel optimization}}}

\author[]{Lahoussine Lafhim$^\dag$ \and  Alain  Zemkoho$^\ddag$}

\thanks{${}^\dag$Laboratoire LASMA, Department of Mathematics, Sidi Mohammed Ben Abdellah University, Morocco (email: \textsf{lahoussine.lafhim@usmba.ac.ma})}%
\thanks{${}^\ddag$School of Mathematical Sciences, University of Southampton, UK (email: \textsf{a.b.zemkoho@soton.ac.uk}).}%

\date{\today}

\begin{abstract}
We consider a multiobjective bilevel optimization problem with vector-valued upper- and lower-level objective functions. Such problems have attracted a lot of interest in recent years. However, so far, scalarization has appeared to be the main approach used to deal with the lower-level problem. Here, we utilize the concept of frontier map that extends the notion of optimal value function to our parametric multiobjective lower-level problem. Based on this, we build a tractable constraint qualification that we use to derive necessary optimality conditions for the problem. Subsequently, we show that our resulting necessary optimality conditions represent a natural extension from standard optimistic bilevel programs with scalar objective functions. 
\end{abstract}
\maketitle

\section{Introduction}
A standard bilevel optimization problem involves the minimization of a real-valued function under a constraint set partly constrained by the optimal solution set of a parametric optimization problem with a scalar objective function; see, e.g., \cite{DempeZemkohoBook} for the most recent surveys on the topic. However, in the last two to three decades, significant attention has being paid to the generalization of this model to the case where the upper- and/or lower-level objective functions are vector-valued. This is precisely the main focus of the work in this paper, as we consider the optimization problem
\begin{equation}\label{optimistic}\tag{MUL}
\mathbb{R}^{p}_+ - {\displaystyle \min_{x,y}} \ F\left( x,y\right) \ \ \text{s.t.} \ x\in X , \ \ y\in S\left( x\right),
\end{equation}
where the function $F :\mathbb{R}^{n}\times\mathbb{R}^{m}\rightarrow \mathbb{R}^{p}$ (with $p\geq 2$) represents the upper-level objective function, while $X \subseteq \mathbb{R}^{n}$ corresponds to the upper-level feasible set. As for the set-valued mapping $S :\mathbb{R}^n \rightrightarrows \mathbb{R}^m$, it collects the optimal solutions of the lower-level problem
\begin{equation}\label{llpill}\tag{L[$x$]}
\mathbb{R}^{q}_+ - {\displaystyle \min_{y}} \ f\left( x,y\right) \ \ \text{s.t.} \ \ y\in Y\left( x\right)
\end{equation}
for a given $x\in \mathbb{R}^n$. Here, $f :\mathbb{R}^{n}\times\mathbb{R}^{m}\rightarrow \mathbb{R}^{q}$ (with $q\geq 2$) represents the lower-level objective function, while the set-valued map $Y :\mathbb{R}^n \rightrightarrows \mathbb{R}^m$ describes the lower-level feasible set.

We associate with the (multiobjective) lower-level problem  \eqref{llpill} the corresponding frontier map $\varPhi :\mathbb{R}^{n} \rightrightarrows \mathbb{R}^{q}$ defined by
\begin{equation}\label{intro1}
\varPhi\left( x\right) := \text{Eff}/\mbox{WEff} \left(f\left( x,Y\left(x\right) \right); \,\mathbb{R}_{+}^{q}\right),
\end{equation}
where the notation $\text{Eff}/\mbox{WEff}$ is used to reflect the fact that optimality in \eqref{intro1} is in the sense of  efficient Pareto (\text{Eff}) or weakly efficient Pareto  ($\mbox{WEff}$) optimality. In the sequel, we will simply write $\varPhi^{E}\left( x\right)=\text{Eff}\left(f\left( x,Y\left(x\right) \right); \,\mathbb{R}_{+}^{q}\right)$ $\left(\mbox{resp}. \varPhi^{W}\left( x\right)=\mbox{WEff} \left(f\left( x,Y\left(x\right) \right); \,\mathbb{R}_{+}^{q}\right)\right)$ when referring to Pareto (resp. weakly Pareto) efficiency in situations where it is necessary to distinguish between these two concepts, which are defined in the next section. Obviously, based on \eqref{intro1}, the set-valued mapping $S : \mathbb{R}^{n} \rightrightarrows \mathbb{R}^{m}$ can be rewritten as
\begin{equation}\label{intro2}
S\left( x\right) :=\left\{y\in Y\left( x\right): \;\; f\left( x,y\right) \in \varPhi\left( x\right)  \right\} \;\, \mbox{ for }\;\,  x\in \mathbb{R}^n.
\end{equation}
Hence, our problem \eqref{optimistic}--\eqref{llpill} can be equivalently written as
\begin{equation}\label{optimistic1}
\mathbb{R}^{p}_+ - {\displaystyle \min_{x,y}} \ F\left( x,y\right) \ \ \text{s.t.} \ x\in X , \;\; y\in Y\left( x\right), \;\;  f\left( x,y\right) \in \varPhi\left(x\right) .
\end{equation}

Note that in the case where $q=1$, meaning that our lower-level problem \eqref{llpill} is simply a standard scalar objective parametric optimization problem, then the frontier map $\varPhi$ reduces to the corresponding optimal value function. And therefore, problem \eqref{optimistic1} will become the standard lower-level value function (LLVF) reformulation well-known in bilevel optimization with scalar objective functions; see, e.g.,  \cite{DDM2007,DemMorZemSIAM,DZ2011,YeZhu1995,YeZhu1995Cor},  for more details on this class of the problem. Hence, clearly, \eqref{optimistic1}  is a natural extension of the LLVF reformulation to the multiobjective bilevel optimization problem \eqref{optimistic}--\eqref{llpill}; thus we labelled it as such throughout this paper.

The number of publications around problem \eqref{optimistic}--\eqref{llpill} or the semivectorial version of the problem, where only the lower-level is multiobjective has been growing significantly over the last decade. Recent surveys on the subject include \cite{Eichfelder2020,SinhaMaloDeb2017}, where overviews of different types of solution algorithms are given. However, our main interest here is on constructing necessary optimality conditions for problem \eqref{optimistic}--\eqref{llpill}; a common point of most works on optimality conditions of this problem is that they rely on some form of scalarization to deal with the multiobjective nature of the lower-level problem \eqref{llpill}; for recent surveys on the subject, see, e.g.,  \cite{vectlower2,DM202} and references therein. 

Additionally, in the latter references, the LLVF reformulation is common after the scalarization step, although  \cite{zemko1} provides a different perspective. Subsequently, as in the case where $p=1$ and $q=1$, the standard approach to develop necessary optimality conditions for the corresponding model, after scalarization, has been the concept of partial calmness introduced in \cite{YeZhu1995}. However, given that partial calmness is, in some sense, equivalent to partial exact penalization of the corresponding value function constraint, it is unclear how such an approach can be directly applied to \eqref{optimistic1} when $p>1$ or $q>1$. Hence, our first main focus in this paper (see Section \ref{Generalized value function constraint qualification}) is to study the possibility to apply the concept of calmness of set-valued mapping, which is closely related to partial calmness \cite{DZ2011,henrionSuroview}. In Section \ref{Generalized value function constraint qualification}, we construct a tractable framework for a set-valued mapping tailored to \eqref{optimistic}--\eqref{llpill} to be used as constraint qualification (CQ) for the problem. In Section \ref{Necessary optimality conditions}, we show how this CQ can be used to develop necessary optimality conditions for problem \eqref{optimistic}--\eqref{llpill}. As a by-product of the regularity condition studied in Section \ref{Necessary optimality conditions}, we provide a new sufficient condition to derive stability for the optimal solution set-valued mapping $S$ \eqref{intro2}; i.e., for the estimation of its coderivative and Lipschitz-likeness.

Before we move to the development of the main results in Sections \ref{Generalized value function constraint qualification} and \ref{Necessary optimality conditions}, in the next section, we first provide some basic variational analysis and multiobjective optimization
concepts that will be needed in the sequel.

\section{Preliminaries}
\subsection{Tools from variational analysis}\label{Tools from var}
In this subsection, we present basic tools from variational analysis that will be used throughout the paper; more on the material covered here can be found in \cite{mordu2,mordu3}, for example. For some point $x\in\mathbb{R}^{n}$ and a scalar $\epsilon > 0$,
\begin{equation*}
	\mathbb{U}_{\epsilon}\left( x\right) :=\{y\in\mathbb{R}^{n}| \ \| y-x\|< \epsilon \}\;\, \mbox{ and }\;  \mathbb{B}_{\epsilon}\left( x\right) :=\{y\in\mathbb{R}^{n}| \ \| y-x\|\leq \epsilon \}
\end{equation*}
denote the open and closed $\epsilon$-ball around $x$, respectively. For brevity, we make use of $\mathbb{U}_{n}=\mathbb{U}_{1}\left( 0\right)$ and $\mathbb{B}_{n}=\mathbb{B}_{1}\left( 0\right)$. For a set-valued mapping $\Upsilon:  \mathbb{R}^{n}\rightrightarrows \mathbb{R}^{m}$, its  Painlev\'e-Kuratowski outer/upper limit at a point $\bar{x}$ is defined by
\begin{equation*}\label{kura}
{\displaystyle \limsup_{x\rightarrow \bar{x}}} \ \Upsilon \left( x\right) :=\left\lbrace x^{*}\in \mathbb{R}^{m}:\;\, \exists x_{k} \rightarrow \bar{x}, \ x^{*}_{k}\rightarrow x^{*}  \ \text{with} \ x^{*}_{k}\in \Upsilon \left( x_{k}\right)  \ \text{for all} \  k\in\mathbb{N} \right\rbrace.
\end{equation*}

Next, consider a set $\Omega\subset \mathbb{R}^{n}$, which is assumed to be closed around a point $\bar{x}\in \Omega$. The Fr\'echet normal cone to $\Omega$ at $\bar{x}\in\Omega$ is defined by
\begin{equation}\label{Frechet Normal}
\widehat{N}\left( \bar{x};\; \Omega\right):=\left\lbrace x^{*}\in \mathbb{R}^{n}:\;\, {\displaystyle \limsup_{\substack{ {x\overset{\Omega }{\rightarrow }\bar{x}} }}} \ \frac{\left\langle x^{*},x-\bar{x}\right\rangle }{\parallel x-\bar{x}\parallel} \leq 0\right\rbrace ,
\end{equation}
where $x\overset{\Omega }{\rightarrow }\bar{x}$ means that $x\rightarrow \bar{x}$ and $x\in\Omega$. Based on this concept, we can introduce the limiting/Mordukhovich normal cone $N\left(\bar{x};\Omega\right)$ to $\Omega$ at $\bar{x}$, which can be obtained by taking the sequential Painlev\'e-Kuratowski upper limits of the Fr\'echet normal cone in \eqref{Frechet Normal}:
\begin{equation*}\label{prem1}
N\left( \bar{x}; \; \Omega\right):= {\displaystyle \limsup_{\substack{ {x\overset{\Omega }{\rightarrow }\bar{x}} }}} \ \widehat{N}\left( x;\Omega\right).
\end{equation*}
If $\bar{x}\notin \Omega$, it is standard to set $N\left( \bar{x}; \;\Omega\right):=\emptyset$. We obviously have $\widehat{N}\left(\bar{x};\Omega\right)\subset N\left(\bar{x}; \; \Omega\right)$ and if the inclusion holds as equality, then we say that $\Omega$ is normally regular at $\bar{x}$. The class of normally regular sets includes convex ones and many other important sets in the field of variational analysis and optimization; see, e.g., \cite{mordu2}, for more details.


Let $\Upsilon: \mathbb{R}^{n}\rightrightarrows \mathbb{R}^{m}$ be a set-valued mapping  with its graph
\begin{equation*}
\text{gph} \ \Upsilon :=\left\lbrace \left( x,y\right) \in \mathbb{R}^{n} \times \mathbb{R}^{m}:\;\, y\in \Upsilon \left( x\right) \right\rbrace,
\end{equation*}
the normal coderivative of $\Upsilon$ at $\left( \bar{x},\bar{y}\right) \in \text{gph} \ \Upsilon$ is defined by
\begin{equation}\label{codDefinition}
D^{\ast}\Upsilon\left( \bar{x},\bar{y}\right)\left( y^{*}\right):=\left\lbrace x^{*}\in \mathbb{R}^{n}:\;\, \left(x^{*},-y^{*}\right) \in N\left( \left( \bar{x},\bar{y}\right); \;\text{gph} \ \Upsilon\right) \right\rbrace \ \ \text{for all} \ y^{*}\in \mathbb{R}^{m}.
\end{equation}
When $\Upsilon$ is a single-valued mapping, to simplify the notation, one writes $D^{\ast}\Upsilon\left( \bar{x}\right)\left( y^{*}\right)$ instead of $D^{\ast}\Upsilon\left( \bar{x},\Upsilon\left( \bar{x}\right) \right)\left( y^{*}\right)$. Furthermore, for a function $f: \mathbb{R}^{n}\rightarrow \mathbb{R}^{m}$ that is strictly differentiable at $\left( \bar{x},\bar{y}\right)$, we have the representation
\begin{equation*}
D^{\ast}f\left( \bar{x}\right)\left( y^{*}\right)=\left\lbrace \nabla f\left( \bar{x}\right)^{\top}y^{*}\right\rbrace \ \ \ \text{for all} \ y^{*}\in \mathbb{R}^{m}.
\end{equation*}

We conclude this subsection with some further properties of set-valued mappings. Consider a set-valued mapping $\Upsilon : \mathbb{R}^{n}\rightrightarrows \mathbb{R}^{m}$. It will be said to be Lipschitz-like around $\left( \bar{x},\bar{y}\right)$ if there exist neighbourhoods $U$ of $\bar{x}$, $V$ of $\bar{y}$, and a constant $l > 0$ such that
		\begin{equation*}
		\Upsilon\left( x\right) \cap V \subseteq \Upsilon\left( u\right) +l\parallel u-x\parallel\mathbb{B}_{m} \ \ \text{for all} \ x,u \in U.
		\end{equation*}
The weaker concept of calmness is said to hold for a set-valued map $\Upsilon$ at some point $\left( \bar{x},\bar{y}\right)$ if there exist neighbourhoods $U$ of $\bar{x}$, $V$ of $\bar{y}$, and a constant $l > 0$ such that
\begin{equation*}
d\left( y,\Upsilon\left( \bar{x}\right)\right) \leq   l \parallel x-\bar{x}\parallel \ \ \text{for all} \ x\in U \ \ \text{and} \ \ y\in V  \cap \Upsilon\left( x\right).
\end{equation*}

Considering the continuous functions $h_{i}:\mathbb{R}^{n}\times\mathbb{R}^{m}\rightarrow \mathbb{R}$ for $i=1, \ldots, q$, we associate the set-valued mapping  $\Upsilon$ defined by
\begin{equation}\label{Uphh}
\Upsilon\left(x\right) := \left\{y\in\mathbb{R}^{m} \ : \ h_{i}\left( x,y\right)  \leq 0, \ i=1,\cdots , q \right\}\;\mbox{ for }\; x\in \mathbb{R}^{n}.
\end{equation}
$\Upsilon$ \eqref{Uphh} will be said to be R-regular at $\left( \bar{x},\bar{y}\right)$ w.r.t. $\Omega\subseteq \mathbb{R}^{n}$ if there are some positive numbers $\sigma$, and $\delta$ such that for all $\left( x,y\right) \in \mathbb{U}_{\delta}\left(\bar{x},\bar{y}\right) \cap \Omega \times \mathbb{R}^{m}$,
\begin{equation}\label{R-regularityMap}
d\left(y, \, \Upsilon\left(x\right) \right) \ \leq \ \sigma \max \left\{0, \;\,\max\left\{h_{i}\left(x, y\right)|\;\, i=1, \cdots ,q\right\}\right\}.
\end{equation}
For more details on R-regularity, see \cite{MehlitzMinchenko2020} and references therein.

A set-valued mapping $\Upsilon: \mathbb{R}^{n}\rightrightarrows \mathbb{R}^{m}$ will be said to be order semicontinuous at a point $\left( \bar{x},\bar{y}\right)\in \text{gph} \ \Upsilon $, if for any sequence $\left( x_{k},y_{k}\right) \in \text{epi} \ \Upsilon$ converging to $\left( \bar{x},\bar{y}\right)$, there is a sequence $\left( x_{k},z_{k}\right) \in \text{gph} \ \Upsilon$ with $y_{k}-z_{k}\in \mathbb{R}_{+}^{m}$ such that $\left( z_{k}\right) $ contains a subsequence converging to $\bar{y}$. Here, $\text{epi} \ \Upsilon $ corresponds to the epigraph of $\Upsilon$ with respect to the ordering cone $\mathbb{R}_{+}^{m}$:
\begin{equation*}
	\text{epi} \ \Upsilon :=\left\lbrace \left( x,y\right) \in \mathbb{R}^{n} \times \mathbb{R}^{m}:\;\, y\in \Upsilon \left( x\right)+\mathbb{R}_{+}^{m} \right\rbrace,
\end{equation*}
Obviously, $\Upsilon$ will be order semicontinuous around $\left( \bar{x},\bar{y}\right)\in \text{gph} \ \Upsilon $ if there exists a neighbourhood $U$ of $\left( \bar{x},\bar{y}\right)$ such that $\Upsilon$ is order semicontinuous at any $\left( x,y\right) \in U\cap \text{gph} \ \Upsilon$.

\subsection{Multiobjective optimization concepts}
Let $C\subset \mathbb{R}^{n}$ be a pointed closed convex cone, with nonempty interior, introducing a partial order denoted  by $\preceq_{C}$ in $\mathbb{R}^{n}$.

\begin{defn}\label{Minimality}
Let $\Omega$ be a nonempty set of $\mathbb{R}^{n}$.  $x\in \Omega$ is said to be a Pareto (resp. weak Pareto)
efficient/minimal vector of $\Omega$ w.r.t. $C$ if
\begin{equation*}\label{parmin}
\Omega\subset x+\left[ \left( \mathbb{R}^{n}\setminus \left( -C\right) \right) \cup \left\lbrace 0\right\rbrace \right]  \ \ \ \left( \text{resp.} \  \Omega\subset x+ \left( \mathbb{R}^{n}\setminus -int C\right) \right),
\end{equation*}
where ``int'' denotes the topological interior of the set in question.
\end{defn}
In the sequel, the set of all the Pareto (resp. weak Pareto) efficient/minimal vectors of a set $\Omega$ w.r.t. $C$ is denoted by $\text{Eff}\left(\Omega;\; C\right)$ (resp. $\text{WEff}\left(\Omega; \;C\right)$).
Let us now consider the following multiobjective optimization problem with respect to the partial order introduced by the pointed, closed and convex cone $C$:
\begin{equation*}\label{rap1}
C-\min f (x) \ \ \text{s.t.} \ x\in \Omega,
\end{equation*}
where $f$ represents a vector-valued function and $\Omega$ the nonempty feasible set. For a
nonempty set $N\subset \Omega$, the image of $N$ by $f$ is defined by
\[
f\left(N\right) :=\left\lbrace f\left( x\right):\;\, x\in N\right\rbrace.
\]
\begin{defn}
A point $\bar{x}\in \Omega$ is said to be a Pareto (resp. weakly Pareto) optimal
solution of problem $\left( \ref{rap1}\right)$ if $f\left( \bar{x}\right)$ is a Pareto (resp. weak Pareto) minimal vector of $f\left( \Omega\right)$, i.e., $f\left( \bar{x}\right)\in \text{Eff}\left( f\left( \Omega\right) ;C\right)$ (resp. $f\left( \bar{x}\right)\in \text{WEff}\left( f\left( \Omega\right) ;C\right)$).
\end{defn}
Similarly, a point $\bar{x}\in \Omega$  is said to be a local Pareto (resp. weakly local Pareto)
optimal solution of problem $\left( \ref{rap1}\right)$ if there exists a neighborhood $U$ of $\bar{x}$ such that $f\left( \bar{x}\right)$ is a Pareto (resp. weak Pareto) minimal vector of $f \left( U\cap \Omega\right)$.
For our analysis of the multiobjective bilevel program \eqref{optimistic}, we will use either the concept of weakly efficient solution for the upper-level problem or the concept of efficient solution for the upper-level problem, and similarly, for the lower-level problem, both notions of efficient optimal solution and weakly efficient solution will be applied.

\section{Generalized value function constraint qualification}\label{Generalized value function constraint qualification}
We start this section by introducing the main constraint qualification that will be used to derive necessary optimality conditions for problem \eqref{optimistic1}. 
\begin{defn}
The generalized value function constraint qualification (GVFCQ) holds at $\left(\bar{x},\bar{y}\right)$ if the set-valued mapping $\Psi : \mathbb{R}^{n}\times\mathbb{R}^{q} \rightrightarrows \mathbb{R}^{n}\times\mathbb{R}^{m}$ defined by
\begin{equation}\label{calm1}
 \Psi\left(u,v\right) :=\left\{\left( x,y\right) \in \text{gph} \ Y: \;\;\left(\begin{array}{c}
                                                                               x \\
                                                                               f(x,y)
                                                                             \end{array}
  \right)+ \left(\begin{array}{c}
                                                                               u\\
                                                                              v
                                                                             \end{array}
  \right) \in \text{gph}~\varPhi \right\},
\end{equation}
 is calm at the point $\left(0, 0, \bar{x},\bar{y}\right)$.
\end{defn}
Note that if the lower-level problem \eqref{llpill} has a scalar objective function, then replacing the frontier map $\varPhi$ by the corresponding optimal value function $\varphi$, then the value function constraint qualification (VFCQ) in this case is obtained by replacing \eqref{calm1} with
\begin{equation}\label{CalPhi2}
\Psi_{\varphi}(v):=\left\{\left( x,y\right) \in \text{gph} \ Y: \;\; f(x,y)-\varphi(x)\leq v\right\}.
\end{equation}
Clearly, we have $\Psi_{\varphi}(v):= \Psi(0, -v)$ if $\text{gph}~\varPhi$ is replaced in \eqref{calm1} by the hypograph of $\varphi$.

It is well-known  that in bilevel programs with scalar objective functions, the VFCQ implies that the partial calmness condition holds in the case where the lower-level feasible set is unperturbed \cite{henrionSuroview}. Moreover, to the best of our knowledge, the VFCQ is the weakest CQ that ensures that partial calmness holds. Hence, since partial cannot be defined for \eqref{optimistic1}, due to the multiobjective nature of the objective functions in \eqref{optimistic}--\eqref{llpill}, it makes sense to consider the GVFCQ as the natural candidate for tractable CQ for the problem under consideration. For the remainder of this section, we focus our attention on constructing sufficient conditions ensuring that GVFCQ can hold.
%

We start with an extension of the uniform weak sharp minimum condition, which enables an extension of a relationship already well-known to be valid in standard bilevel optimization problems with scalar objective functions \cite{henrionSuroview,Ye1998,YeZhu1995,YeZhu1995Cor}.
\begin{defn}\label{def-w-1010} The local uniform weak sharp minimum (LUWSM) condition holds at $\left( \bar{x},\bar{y}\right)$,
 for the family problems \eqref{llpill}$_{x\in X}$, if there exist $\epsilon > 0$ and $\lambda > 0$  such that
\begin{equation*}\label{secondapp}
\forall \left( x,y\right) \in V_{\epsilon}\left( \bar{x},\bar{y}\right) : \ \ y\in Y\left( x\right)\; \Longrightarrow \; \lambda d\left( y, \,S\left( x\right) \right) \leq d\left(f\left(x, y\right);\; \varPhi\left( x\right)\right).
\end{equation*}
\end{defn}
If $V_{\epsilon}\left( \bar{x},\bar{y}\right)$ is replaced by the whole space $\mathbb{R}^{n}\times \mathbb{R}^{m}$ in this definition, we simply say that the uniformly weak sharp minimum (UWSM) condition holds at $\left( \bar{x},\bar{y}\right)$.
\begin{thm}
Let $\left( \bar{x},\bar{y}\right)\in \text{gph} S$ and $f$ be locally Lipschitzian around $\left( \bar{x},\bar{y}\right)$ with constant $L$ and assume that $\varPhi$ is Lipschitz-like around $\left( \bar{x},\bar{z}\right)$, with $\bar{z}=f\left( \bar{x},\bar{y}\right)$. If the LUWSM condition holds at $\left( \bar{x},\bar{y}\right)$, then the GVFCQ is satisfied at $\left(\bar{x},\bar{y}\right)$.
\end{thm}
\begin{proof}
Based on the assumptions, there exit $l > 0$ and $\delta >0$ such that
	\begin{equation}\label{sharp-ii-1}
	\varPhi\left( x_{0}\right) \cap \left( \bar{z}+\delta \mathbb{B}\right)   \subset \varPhi\left( x_{1}\right) +l \parallel x_{0}-x_{1}\parallel \mathbb{B} \ \ \text{for all} \ x_{0}, \ x_{1}\in \bar{x}+\delta \mathbb{B}_{q}.
	\end{equation}
Let $0< \epsilon < \frac{\delta}{2}$ and $\lambda > 0$ be the constants from Definition \ref{def-w-1010} and $u \in \epsilon \mathbb{B}_{n}$, $v \in \epsilon \mathbb{B}_{q}$, and $\left( x,y\right) \in \left( \bar{x},\bar{y}\right)+\epsilon \mathbb{B}_{n\times m}$ such that $\left( x,y\right)\in\Psi\left( u,v\right)$. Since $\Psi\left( 0,0\right)=\text{gph}S$, then
\begin{equation}\label{sharp-ii}
d\left( \left( x,y\right),\Psi \left( 0,0\right)\right) =d\left( \left( x,y\right),\text{gph}S\right) \leq d\left( y,S\left( x\right) \right).
\end{equation}
By the local uniform weak sharp minimum condition, we have
	\begin{equation}\label{sharp-ii-11}
d\left( y,S\left( x\right) \right) \ \leq \ \lambda^{-1} d\left(f\left( x,y\right); \varPhi\left( x\right)\right).
    \end{equation}
Since, $f$ is locally Lipschitzian around $\left( \bar{x},\bar{y}\right)$ with constant $L$ and radius $\alpha$, then setting $\beta =\min~\left\{\alpha , \frac{\delta}{4L}\right\}$ leads to
\begin{equation*}
	\begin{array}{lcl}
		\parallel v+f\left( x,y\right) -f\left( \bar{x},\bar{y}\right) \parallel & \leq & \parallel v\parallel + L\left(  \parallel x-\bar{x}\parallel +\parallel y-\bar{y}\parallel\right), \\
		& \leq & \epsilon +L\left( \frac{\delta}{4L}+\frac{\delta}{4L}\right), \\
		& \leq & \frac{\delta}{2} +\frac{\delta}{2}, \\
		& = & \delta
	\end{array}
	\end{equation*}
for all $\left( x,y\right)\in \left( \bar{x},\bar{y}\right)+ \beta\mathbb{B}_{n\times m}$. Thus, $v+f\left( x,y\right) \in \bar{z}+\delta\mathbb{B}_{q}$. Taking $x_{0}=x+u$ and $x_{1}=x$ while considering that $x+u\in \bar{x}+\delta \mathbb{B}_{n}$, it follows from $\left( \ref{sharp-ii-1}\right)$ that there exist $z\in \varPhi\left( x\right)$ such that
\begin{equation*}
\parallel v+f\left( x,y\right)-z \parallel \leq l   \parallel u\parallel .
\end{equation*}
Consequently,
\begin{equation}\label{sharp-ii-12}
d\left(f\left( x,y\right); \varPhi\left( x\right)\right) \leq \parallel f\left( x,y\right)-z \parallel \leq l   \parallel u\parallel +\parallel v\parallel .
\end{equation}
Setting $\tau =\max\left( l,1\right)$ and combining $\left( \ref{sharp-ii}\right)$, $\left( \ref{sharp-ii-11}\right)$, and $\left( \ref{sharp-ii-12}\right)$, it follows that
\[
d\left( \left( x,y\right),\Psi \left( 0,0\right)\right) \leq \lambda^{-1}\tau \parallel \left( u,v\right) \parallel
\]
for all $\left( u,v\right) \in \epsilon \mathbb{B}_{n\times q}$ and $\left( x,y\right)\in \Psi\left( u,v\right)\cap
\left( \left( \bar{x},\bar{y}\right)+ \epsilon \mathbb{B}_{n\times m}\right)$. Hence, the result.
\end{proof}

To provide a concrete case where the LUWSM condition holds, we consider the parametric linear multiobjective optimization problem
\begin{equation}\label{linbil}
{\displaystyle \mathbb{R}_{+}^{q}-\min_{y}} \ Cy \ \ \ \text{ s.t. } \;\ Ax+By\leq d,
\end{equation}
where $d\in\mathbb{R}^{k}$, \ $C \in\mathbb{R}^{q}\times\mathbb{R}^{m}$, \ $A \in\mathbb{R}^{k}\times\mathbb{R}^{n}$ and $B \in\mathbb{R}^{k}\times\mathbb{R}^{m}$. To state the corresponding result, let $\Gamma$ denotes the simplex defined by
\begin{equation}\label{Delta}
\Gamma:=\left\{\alpha \in\mathbb{R}^{q}: \;\; \alpha\geq 0, \ \ {\displaystyle \sum^{q}_{i=1}} \ \alpha_{i}=1\right\}.
\end{equation}
\begin{prop}\label{secondrela3}Consider a family of problems \eqref{llpill}$_{x\in X}$ defined in \eqref{linbil} with $X \subseteq \mathbb{R}^n$, and let the corresponding version of the set-valued mapping $S$ \eqref{intro2} for problem \eqref{linbil} be uniformly bounded on $X$; i.e., there exits some $k>0$ such that for all  $x\in X$ and $y\in S\left( x\right)$, $\parallel y\parallel \leq k$. Furthermore, suppose that there exists a constant $\delta > 0$ such that for all $\alpha \in \Gamma$, $x\in X$, and $y\in S\left( x\right)$, we have $\alpha^{T}Cy \geq \delta$. Then the UWSM condition holds. 
\end{prop}
\begin{proof}
Let $x\in X$ and consider the family of sets
\begin{equation}\label{rela0}
S_{\alpha}\left( x\right) := \arg\underset{y}\min \left\{\alpha^{T}Cy: \;\; Ax+By\leq d \right\}\;  \mbox{ for } \; \alpha\in\Gamma.
\end{equation}
Given that the set-valued mapping $G\left( x\right) =\{y\in\mathbb{R}^{m} \ : \ Ax+By\leq d  \}$ is a polyhedral and convex-valued, it follows from \cite{ABB} (see also \cite[Theorem 3.3, pp 96]{luc}) that there are finitely many vectors $\alpha_{1}\left(x\right)$,\ldots, $\alpha_{s}\left(x\right)$ of the set $\Gamma$ \eqref{Delta} such that we have
\begin{equation}\label{rela}
S\left( x\right) \ = \ {\displaystyle \bigcup^{s}_{j=1}} \ S_{\alpha_{j}\left( x\right) }\left( x\right).
\end{equation}
Let $y\in G\left( x\right)$. If $y\in S\left( x\right)$, then $\left( \ref{secondapp}\right)$ holds true. Otherwise, considering any $y\in G\left( x\right)\setminus S\left( x\right)$, we have $0\notin Cy -\varPhi \left( x\right)$. Now, let $z\in   \varPhi \left( x\right)$, then there is some $\tilde{y}\in S\left( x\right)$ such that
$z=C\tilde{y}$ and $Cy -C\tilde{y}\neq 0$.
On the other side, setting $a=\alpha_{j}^{T}\left( x\right)C$ and $b=\alpha_{j}^{T}\left( x\right)C\tilde{y}$ and using Hoffman's lemma (see \cite[Theorem 1]{hoffman1}) it follows from $\left( \ref{rela0}\right)$ and $\left( \ref{rela}\right)$
\[
d\left( y,S\left( x\right) \right)  \leq  d\left( y,S_{\alpha_{j}}\left( x\right) \right) \leq  k \delta^{-1} \alpha_{j}^{T}\left( x\right) \left( Cy -C\tilde{y}\right),
\]
where $k$ is the constant appearing in uniform boundedness of $S$. Hence,
\[
\begin{array}{lcl}
d\left( y,S\left( x\right) \right) & \leq &  k \delta^{-1} \ \parallel \alpha_{j}\left( x\right)\parallel_{1} \ \parallel Cy - C\tilde{y} \parallel, \\
& \leq & \lambda^{-1}  \ \parallel Cy - C\tilde{y}  \parallel,
\end{array}
\]
where $\lambda^{-1} =k \delta^{-1}$ and $\parallel \alpha_{j}\left( x\right)\parallel_{1}=1$. It follows from the last inequality that
\[
Cy - C\tilde{y} \notin \lambda d\left( y,S\left( x\right)  \right)\mathbb{U}_{q}.
\]

Finally, since $z=C\tilde{y}$ is arbitrary, we get
$
Cy -\varPhi\left( x\right) \cap \lambda d\left( y,S\left( x\right) \right)\mathbb{U}_{q} =\emptyset.
$
This means that for all $z\in \varPhi\left( x\right)$,
$
Cy - z \notin \lambda d\left( y,S\left( x\right) \right) \mathbb{U}_{q}.
$
Consequently,
$
\lambda d\left( y,S\left( x\right) \right)  \leq  \parallel Cy - z\parallel
$
for all $z\in \varPhi\left( x\right)$. This implies that
$
\lambda d\left( y, S\left( x\right) \right)  \leq  d\left(Cy, \varPhi\left( x\right)\right).
$
Hence, the result.
\end{proof}

Next, we provide an example where all the assumptions of proposition \ref{secondrela3} are satisfied. 
\begin{exmp}\label{FirstExample}
Setting $X:=[4, \;\infty)\times [3, \infty)$ and considering problem \eqref{linbil} with
\begin{equation}\label{MaterialExample}
C:=\left(\begin{array}{cc}
           2 & 0 \\
           0 & 1
         \end{array}
 \right),
 \quad
 A:=\left(\begin{array}{rr}
           0 & 0 \\
           0 & 0\\
           0 & 0\\
           0 & 0\\
           -1& 0\\
           0 & -1
         \end{array}\right),
 \quad
 B:=\left(\begin{array}{rr}
           1 & 0 \\
           -1 & 0\\
           0 & 2\\
           0 & -1\\
           1& 0\\
           0 & 1
         \end{array}\right),
 \quad \mbox{and }\quad
d:=\left(\begin{array}{r}
          4\\
          -1\\
          6\\
          -2\\
          0\\
          0
         \end{array}\right),
\end{equation}
we can easily check that for any $x\in X$ and $y$ such that $Ax + By \leq d$, taking any $(\mu, \nu)\in \mathbb{R}^2_+$ such that $\mu + \nu =1$, we have the inequality
\[
(\mu, \nu)C y = 2\mu y_1 + \nu y_2\geq 2.
\]
\end{exmp}

In case the uniform boundedness of the set-valued mapping $S$ required in Proposition \ref{secondrela3} is not satisfied, we can use the following alternative result. 
\begin{prop}\label{firstapp11}
Consider a family of problems \eqref{llpill}$_{x\in X}$ defined in \eqref{linbil} with $X \subseteq \mathbb{R}^n$ such that for all $x\in X$ and  $j\in \{1,\cdots ,s\}$, the sets $S_{\alpha_{j}}\left( x\right)$ from  \eqref{rela} are unbounded. Furthermore, suppose that there exist $\delta > 0$  and a unit vector $z\in\mathbb{R}^{q}$ such that for all a constant $\alpha \in \Gamma$ and  $x\in X$, we have  $\alpha^{T}Cz\geq \delta > 0$. Then the UWSM condition holds. 
\end{prop}
\begin{proof}
Its folows on the path of  Proposition \ref{secondrela3}. We shall argue in the same way as above and use \cite[Theorem 2]{hoffman1} instead of \cite[Theorem 1]{hoffman1}. 
\end{proof}

The next result provides a sufficient condition for the existence of uniform weak sharp minimun tailored to a more general multiobjective bilevel optimization problem.


\begin{thm}\label{suff-ye} The UWSM condition holds for any general family of problems \eqref{llpill}$_{x\in X}$, where $f$ is Lipschitz continuous in $y$ uniformly in $x\in X$, the set $Y\left( x\right)$ is closed for any fixed $x\in X$, and there exists a strictly positive number $\lambda$ such that we have
\[
\begin{array}{ll}
  \parallel \varsigma \parallel \geq \lambda^{-1}, & \forall \varsigma \in \partial_{y} \left\langle y^{\ast},f \right\rangle \left( x,y\right)  +N\left( y,Y\left( x\right)\right),\\[1ex]
                                                   & \forall y^{\ast}\in N\left( f\left( x,y\right), \; z-\mathbb{R}_{+}^{q}\right), \;\; z\in \varPhi\left( x\right), \;\; \left( x,y\right) \in \text{gph} \ Y, \;\; y\notin S\left( x\right).
\end{array}
\]
\end{thm}
\begin{proof}
Consider any closed subset $\Lambda$ of $\mathbb{R}^{m}$, a locally Lipschitz function $\phi :\mathbb{R}^{m}\rightarrow \mathbb{R}^{q}$ with constant $L$,  a vector $z\in\mathbb{R}^{q}$, and the set
\begin{equation*}\label{ass1}
\Xi\left( \phi ,z\right) =\left\{y\in \Lambda \ : \ \phi\left( y\right) \leq z\right\}
\end{equation*}
and the function
\begin{equation*}\label{ass2}
\phi_{z}^{+}\left( y\right) = d\left( \phi\left( y\right) ,z-\mathbb{R}_{+}^{q}\right) = {\displaystyle \max^{q}_{i=1}} \ \left(\phi_{i}\left( y\right) -z_{i}\right)_{+} ,
\end{equation*}
where the distance function is defined by the max norm on $\mathbb{R}^{q}$ and $a_{+}=\max \{a,0\}$. Now, let us show that if there exist $\lambda > 0$ and $0< \epsilon \leq +\infty$   such that
\begin{equation}\label{suf-eq1}
\parallel \varsigma \parallel \geq \lambda^{-1}
\end{equation}
for all  $\varsigma \in \partial \left\langle y^{\ast},\phi \right\rangle \left( y\right) +N\left( y, \Lambda\right)$,
$y^{\ast}\in N\left( \phi\left( y\right),z-\mathbb{R}_{+}^{q}\right)$,
$y\in \Lambda$, and $0< \phi_{i}\left( y\right)-z_{i} < \epsilon$ for some $i$, then we have
\begin{equation}\label{suf-eq2}
d\left( y,\Xi\left( \phi ,z\right)\right)\leq \lambda  \phi_{z}^{+}\left( y\right), \ \ \ \forall y\in \Lambda \ \ \text{such that} \ \ \phi_{z}^{+}\left( y\right) < \epsilon\left( 1+L\lambda\right)^{-1} .
\end{equation}

First, by contraposition, suppose that there exist $\bar{y}\in \Lambda$ such that
\begin{equation}\label{absurd}
\lambda \phi_{z}^{+}\left( \bar{y}\right)  < d\left( \bar{y},\Xi\left( \phi ,z\right)\right) \ \ \ \text{and} \ \ \ \phi_{z}^{+}\left( \bar{y}\right) < \epsilon\left( 1+L\lambda\right)^{-1}.
\end{equation}
It is obvious, by choosing suitable $r> 1$, that the following inequalities hold
\begin{equation*}
\delta < d\left( \bar{y},\Xi\left( \phi ,z\right)\right) \ \ \ \text{and} \ \ \ \phi_{z}^{+}\left( \bar{y}\right) < \epsilon\left( 1+rL\lambda\right)^{-1}
\end{equation*}
with $\delta =r \lambda \phi_{z}^{+}\left( \bar{y}\right)$. Now, observing that
\begin{equation*}
\phi_{z}^{+}\left( \bar{y}\right) \leq {\displaystyle \inf_{y\in \Lambda}} \ \phi_{z}^{+}\left( y\right)+ \delta \left( r\lambda\right)^{-1},
\end{equation*}
one can deduce that
\begin{equation*}
\psi\left( \bar{y}\right) \leq {\displaystyle \inf_{y\in \Lambda}} \ \psi\left( y\right)+ \epsilon
\end{equation*}
with $\psi\left( y\right) = \phi_{z}^{+}\left( y\right)+\delta_{\Lambda}\left( y\right) $, $\delta_{\Lambda}$ is the indicator function of the set $\Lambda$ and $\epsilon =\delta \left( r\lambda\right)^{-1}$. Hence, applying the variational principle of Ekeland we find $v\in\Lambda$ such that
\begin{equation}\label{ekland}
\left\{
\begin{array}{l}
\parallel v-\bar{y} \parallel \leq \delta ,\\
\psi\left( v\right)\leq \psi\left( y\right)+ \left( r\lambda\right)^{-1} \parallel y-v \parallel \ \ \ \text{for all} \ \ y\in\Lambda .
\end{array}
\right.
\end{equation}
Hence, $v$ is a minimum of the function $y \longmapsto \psi\left( y\right) +\left( r\lambda\right)^{-1} \parallel y-v \parallel$ and we get, by exploiting the chain rule, that
\begin{equation}\label{evp}
0\in \partial \phi_{z}^{+}\left( v\right)+ N\left( v, \Lambda\right)  +\left( r\lambda\right)^{-1} \mathbb{B}_{m}.
\end{equation}
In view of \cite[Theorem 1.97 and Corrolary 3.43]{mordu2} it follows that
\begin{equation*}
\partial \phi_{z}^{+}\left( v\right) \ \subset \ {\displaystyle \bigcup_{y^{\ast}\in N\left( \phi\left( v\right) ,z-\mathbb{R}_{+}^{q}\right) }} \ \partial \left\langle y^{\ast},\phi \right\rangle \left( v\right) .
\end{equation*}
Consequently, there exist $y^{\ast}\in N\left( \phi\left( v\right),z-\mathbb{R}_{+}^{q}\right)$ and $\varsigma \in \partial \left\langle y^{\ast},\phi \right\rangle \left( v\right)+ N\left( v,\Lambda\right)$ such that $\left( \ref{evp}\right) $ yields
\begin{equation*}
\parallel \varsigma\parallel \leq \left( r\lambda\right)^{-1} < \lambda^{-1}.
\end{equation*}
According to $\left( \ref{absurd}\right)$, $\left( \ref{ekland}\right)$ and $v\in \Lambda$, we have $v\notin  \Xi\left(\phi ,z \right)$. Consequently, $0< \phi_{i}\left( v\right)-z_{i} < \epsilon$ for some $i$. On the other hand, since $\parallel v-\bar{y} \parallel \leq \delta$, the condition $\left( \ref{absurd}\right)$ guarantees that
\begin{equation*}
\begin{array}{lcl}
\phi_{z}^{+}\left( v\right) & \leq & \phi_{z}^{+}\left( \bar{y}\right) + L \parallel v-\bar{y} \parallel, \\
& \leq & \phi_{z}^{+}\left( \bar{y}\right) + L \delta, \\
& \leq & \phi_{z}^{+}\left( \bar{y}\right)\left( 1+Lr\lambda\right), \\
& \leq & \epsilon \left( 1+Lr\lambda\right)^{-1}\left( 1+Lr\lambda\right), \\
& \leq & \epsilon.
\end{array}
\end{equation*}
Since $\phi_{i}\left( v\right)-z_{i} \leq \phi_{z}^{+}\left( v\right)$, we deduce that $\parallel \varsigma\parallel \leq \left( r\lambda\right)^{-1} < \lambda^{-1}$ and $\phi_{i}\left( y\right)-z_{i} \leq  \epsilon$, which contradict $\left( \ref{suf-eq1}\right) $ and justifies the required inclusion $\left( \ref{suf-eq2}\right) $.

Secondly, taking $\phi \left( y\right) =f\left( x,y\right)$, \ $\Lambda =Y\left( x\right)$, $z\in \varPhi\left( x\right)$, and observing that
\begin{equation*}
\Xi\left( \phi,z\right)=\{y\in Y\left( x\right) \ : \ f\left( x,y\right) \leq  z \} \ \subset \ S\left( x\right),
\end{equation*}
it holds that
\begin{equation*}
\begin{array}{lcll}
d\left( y,S\left( x\right) \right) & \leq & d\left( y,\Xi\left( z,f\right)\right),  &  \\
& \leq & \lambda d\left( f\left( x,y\right) , z-\mathbb{R}_{+}^{q}\right), & \\
& \leq & \lambda d\left( f\left( x,y\right) , z\right).
\end{array}
\end{equation*}
Since, $z$ is arbitrary in $\varPhi\left( x\right)$, then
$d\left( y,S\left( x\right) \right)  \leq  \lambda  d\left( f\left(x, y\right) , \varPhi\left( x\right) \right).$
\end{proof}
\begin{figure}[htp]\label{Figure1}
  \begin{center}
  	\begin{tikzpicture}[sharp corners=2pt,inner sep=5pt,node distance=.6cm,every text node part/.style={align=center}]
  	
  	\node[draw, minimum height = 1cm, minimum width = 4cm] (state0){Linear CQ
  	};
  	\node[draw,right=2cm of state0, minimum height = 1cm, minimum width = 4cm](state1){NonLinear CQ
  	};

  	\node[draw,below=2cm of state0, minimum height = 1cm, minimum width = 4cm](state2){UWSM};
  	
    \node[draw,below=2cm of state1, minimum height = 1cm, minimum width = 4cm](state3){GVFCQ
    };
  	
  	\node[draw,below=2cm of state2, minimum height = 1cm, minimum width = 4cm](state4){LUWSM};
  	
  	\node[draw,right=2cm of state4, minimum height = 1cm, minimum width = 4cm](state5){R-regularity
  	};
  	
  	
  	\draw[-triangle 60] (state0) -- (state2) node [midway, above, left = 0.1cm]{};
  	\draw[-triangle 60] (state5) -- (state4) node [midway, above]{};
  	\draw[-triangle 60] (state2) -- (state4) node [midway, above]{};
  	\draw[-triangle 60] (state4) -- (state3) node [midway, above]{};
  	\draw[-triangle 60] (state1) -- (state2) node [midway, above, rotate = -45]{};
  	\end{tikzpicture}
  	  \end{center}
\caption{
	Linear CQ refers to the assumptions in Proposition \ref{secondrela3} or Proposition \ref{firstapp11}, while NonLinear CQ represents the assumptions in Theorem \ref{suff-ye}.}
\label{Figure1}
\end{figure}
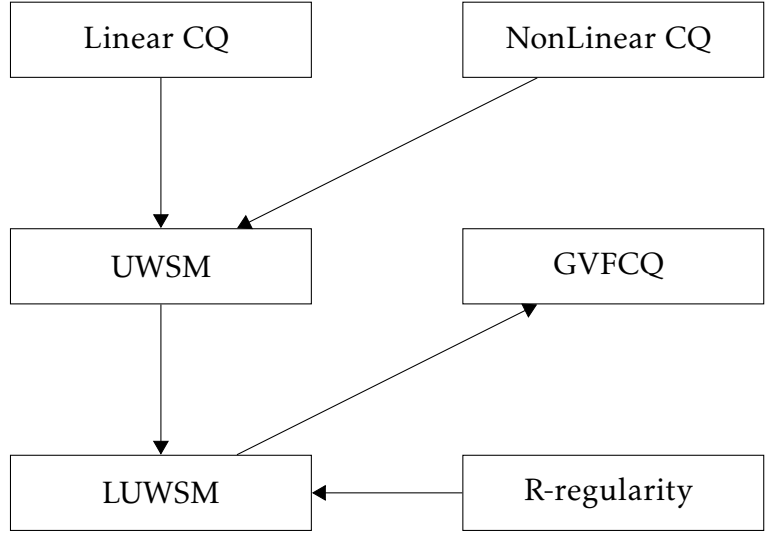

To conclude this section, we provide another sufficient condition for the LUWSM condition based on the R-regularity concept introduced in Subsection \ref{Tools from var}. To proceed, observe that the lower-level efficient solution mapping $S$ \eqref{intro2} can be rewritten as
\[
S(x)=\left\{y\in \mathbb{R}^m:\;\;   d\left( f\left( x,y\right) ,\varPhi\left( x\right)\right)\leq 0, \;\; d\left( \left( x,y\right),\text{gph} \ Y\right)\leq 0\right\}.
\]
Hence, we will say that the R-regularity constraint qualification (RRCQ) holds at the point $\left( \bar{x}, \bar{y}\right)\in \text{gph} \ S$ if $S$ is R-regular \eqref{R-regularityMap} at $\left( \bar{x}, \bar{y}\right)$ w.r.t. $\text{dom} \ S$.
\begin{prop}
If RRCQ holds at $\left( \bar{x}, \bar{y}\right)$ and there is some neighborhood $U\subset\mathbb{R}^{n}$ of $\bar{x}$ such that $\text{dom} \ Y\cap U=  \text{dom} \ S\cap U$, then LUWSM is satisfied at $\left( \bar{x}, \bar{y}\right)$.
\end{prop}
\begin{proof}
Fix $\left( \bar{x}, \bar{y}\right) \in  \text{gph} \ S$. Since, the mapping $S$ is R-regular at $\left( \bar{x}, \bar{y}\right)$ w.r.t. $\text{dom} \ S$, there exist  $\sigma > 0$ and $\epsilon > 0$ such that for all $\left( x, y\right) \in \mathbb{U}_{\epsilon}\left( \bar{x}, \bar{y}\right)\cap \left( \text{dom} \ S \times \mathbb{R}^{m}\right) $ we have the inequality
\begin{equation*}
d\left( y,S\left( x\right) \right) \ \leq \ \sigma \ \max \{0,d\left( f\left( x,y\right) ,\varPhi\left( x\right)\right) ,d\left( \left( x,y\right),\text{gph} \ Y\right) \}.
\end{equation*}
From the definition of the frontier map, for any $\left( x,y\right) \in \mathbb{U}_{\epsilon}\left( \bar{x}, \bar{y}\right)$ with $d\left( \left( x,y\right),\text{gph} \ Y\right) = 0$, we  have the inequality $d\left( f\left( x,y\right) ,\varPhi\left( x\right)\right)  \geq 0$. Hence, for all $\left( x,y\right) \in \mathbb{U}_{\epsilon}\left( \bar{x}, \bar{y}\right)\cap \left( \text{dom} \ S \times \mathbb{R}^{m}\right) $, 
\begin{equation*}\label{reg}
y\in Y\left( x\right) \Longrightarrow d\left( y,S\left( x\right) \right) \ \leq \ \sigma \ d\left( f\left( x,y\right) ,\varPhi\left( x\right)\right).
\end{equation*}
On the other hand, we can choose an open ball around $\left( \bar{x}, \bar{y}\right)$ which is contained in $\mathbb{U}_{\epsilon}\left( \bar{x}, \bar{y}\right)\cap \left( \text{dom} \ S \times \mathbb{R}^{m}\right)$. Dividing $\left( \ref{reg}\right)$ by $\sigma$, we get the result.
\end{proof}
Finally, note that the relationships between all the constraint qualifications discussed above are summarized in Figure \ref{Figure1}.

\section{Necessary optimality conditions}\label{Necessary optimality conditions}
Our aim in this section is to use the GVFCQ, introduced and studied in the previous section, to derive necessary optimality conditions for problem \eqref{optimistic}--\eqref{llpill}. To proceed, we consider the set
\begin{equation}\label{mapmapmap}
\Pi :=\left( X\times \mathbb{R}^{m}\right) \cap \text{gph} \ S \subset \mathbb{R}^{n}\times\mathbb{R}^{m}
\end{equation}
and the set-valued mapping
\begin{equation}\label{immap}
\Sigma\left( x\right):= f\left( x,Y\left( x\right) \right) := \left\{f\left( x,y\right): \;\, y\in Y\left( x\right) \right\}\;\, \mbox{ for }\;\, x\in\mathbb{R}^{n}.
\end{equation}
In the process, the upper estimates for coderivatives of the optimal solution set-valued mapping $S$ and  the frontier map $\varPhi$ will also be useful. To specifically compute an estimate of the coderivative of $\varPhi^{E}$ (see \eqref{intro1} and related discussion),  we additionally need the following \emph{strong domination property} for the lower-level problem \eqref{llpill}:
\begin{equation}\label{dompro}
f\left( x,Y\left( x\right) \right) \subset \varPhi^{E}\left( x\right)  + \mathbb{R}_{+}^{q} \ \ \ \ \ \forall x\in X.
\end{equation}
This property has been used in the literature with different names; for example, it is used in \cite{Tanino}, where is called $\mathbb{R}_{+}^{q}$-minicomplete property, and utilized to estimate  the contingent derivative of the set-valued mapping $\Sigma$. However, we borrow our vocabulary from the following weaker \emph{domination property} used in \cite{mordu1}:
\[
f\left( x,Y\left( x\right) \right) +\mathbb{R}_{+}^{q} = \varPhi^{E}\left( x\right)  + \mathbb{R}_{+}^{q}.
\]

To construct an estimate of the coderivative of $\varPhi^E$ in the next result, we also need the limiting qualification condition at a reference point $\left(\bar{x},\bar{y}\right)$:
\begin{equation}\label{mqc}
D^{*}S\left( \bar{x},\bar{y}\right) \cap \left( -N\left( \bar{x}, X \right) \right) =\{0\}.
\end{equation}
\begin{thm}\label{estime-co-F}
	Let $\left(\bar{x},\bar{y}\right) \in \text{gph} \ S$. Suppose that $f$ is locally Lipschitzian around $\left( \bar{x},\bar{y}\right)$, that the graph of the image map $\Sigma$ is locally compact around $\bar{x}$, that $Y$ is locally closed around $\left( \bar{x},\bar{y}\right)$ with $\bar{z}=f\left( \bar{x},\bar{y}\right)$ and the strong domination property \eqref{dompro} is satisfied. Suppose in addition that $Y$ is Lipschitz-like around $\left( \bar{x},\bar{y}\right)$. Then, it holds that
	\begin{equation}\label{DEBobo}
	D^{*}\varPhi^{E}\left(\bar{x},\bar{z}\right) \left( z^{*}\right) \subset {\displaystyle \bigcup_{\left( x^{*},y^{*}\right) \in D^{*}f\left(\bar{x},\bar{y}\right) \left( z^{*}\right)}} \ [x^{*}+D^{*} Y\left( \bar{x},\bar{y}\right) \left( y^{*}\right)], \ \ \ \text{for all} \ z^{*}\in\mathbb{R}^{q}
	\end{equation}
and $\varPhi^{E}$ is Lipschitz-like around $\left(\bar{x},\bar{z}\right)$. Furthermore, if the function $f$ is strictly differentiable at $\left( \bar{x},\bar{y}\right)$, then for any $z^{*}\in\mathbb{R}^{q}$, we have
	\begin{equation*}
	D^{*}\varPhi^{E}\left( \bar{x},\bar{z}\right) \left( z^{*}\right) \subset \nabla_{x}f\left( \bar{x},\bar{y}\right)^{*} z^{*}+ D^{*} Y\left( \bar{x},\bar{y}\right) \left(\nabla_{y}f\left( \bar{x},\bar{y}\right)^{*} z^{*}\right) .
	\end{equation*}
\end{thm}
\begin{proof}
	First, observe that the image map $\Sigma$ in $\left( \ref{immap}\right)$   has a composite form. Hence, applying to this composition the coderivative chain rule from \cite[Theorem 3.18(i)]{mordu2} for the locally Lipschitzian
	cost mapping $f\left( x,y\right)$, we get
	\begin{equation}\label{aubin-1}
	D^{*}\Sigma\left( \bar{x},\bar{z}\right) \left( z^{*}\right) \subset {\displaystyle \bigcup_{\left( x^{*},y^{*}\right) \in D^{*}f\left( \bar{x},\bar{y}\right) \left( z^{*}\right)}} \ [x^{*}+D^{*} Y\left( \bar{x},\bar{y}\right) \left( y^{*}\right)], \ \ \ z^{*}\in\mathbb{R}^{q} .
	\end{equation}
	Fix $z^{*}\in\mathbb{R}^{q}$ and let us prove that $D^{*}\varPhi^{E}\left( \bar{x},\bar{z}\right) \left( z^{*}\right) \subset D^{*}\Sigma\left( \bar{x},\bar{z}\right) \left( z^{*}\right)$. Let $x^{*}\in D^{*}\varPhi^{E}\left( \bar{x},\bar{z}\right) \left( z^{*}\right)$. Based on \eqref{codDefinition}, there are sequences $\substack{ {\left( x_{k},z_{k}\right)\overset{\text{gph} \ \varPhi^{E}}{\rightarrow }\left( \bar{x},\bar{z}\right)} }$ and $\left( x^{*}_{k},z^{*}_{k}\right) \rightarrow \left(x^{*},z^{*} \right)$ such that
	 \begin{equation*}
	 {\displaystyle \limsup_{\substack{ {\left( x_{k_{s}},z_{k_{s}}\right)\overset{\text{gph} \ \varPhi^{E}}{\rightarrow }\left(x_{k},z_{k}\right)} }}} \ \frac{\left\langle x^{*},x_{k_{s}}-x_{k}\right\rangle  -\left\langle z^{*},z_{k_{s}}-z_{k}\right\rangle}{\parallel x_{k_{s}}-x_{k}\parallel + \parallel z_{k_{s}}-z_{k}\parallel} \leq 0.
	 \end{equation*}
	We claim that in some neighborhood $U$ of $\left( \bar{x},\bar{z}\right)$ for any $\left( x_{k_{s}},z_{k_{s}}\right)\in U$ such that
\[
\substack{ {\left( x_{k_{s}},z_{k_{s}}\right)\overset{\text{gph} \ \varPhi^{E}}{\rightarrow }\left(x_{k},z_{k}\right)}}, \;\,\mbox{ one has }\;\,\substack{ {\left( x_{k_{s}},z_{k_{s}}\right)\overset{\text{gph} \ \Sigma}{\rightarrow }\left(x_{k},z_{k}\right)} }.
\]
Indeed, suppose contrary to our claim, that there exists
\[
\left( x_{k_{s}},z_{k_{s}}\right)\in \text{gph} \ \Sigma\setminus \text{gph} \ \varPhi^{E}\; \mbox{ such that } \;\left( x_{k_{s}},z_{k_{s}}\right)\rightarrow \left(x_{k},z_{k}\right).
\]
It follows immediately from the strong domination property \eqref{dompro} that $\left( x_{k_{s}},z_{k_{s}}\right)\in \text{epi} \ \varPhi^{E}$. Since  $\Sigma$ is locally compact around $\bar{x}$ it follows from \cite[Proposition 4.3 (iv)]{li-xue} that $\varPhi^{E}$ is locally order semicontinuous around $\left( \bar{x},\bar{z}\right)$. Hence, for $\left( x_{k_{s}},z_{k_{s}}\right)\in \text{epi} \ \varPhi^{E}$, there exists a sequence $\left( x_{k_{s}},t_{k_{s}}\right)\in \text{gph} \ \varPhi^{E}$ such that $z_{k_{s}} \in t_{k_{s}} +\mathbb{R}_{+}^{q}$. Then applying \cite[Theorem 1.3]{levy}, we get a contradiction while considering the fact that $D^{*}\Sigma\left( \bar{x},\bar{z}\right) \left( 0\right)=\{0\}$, which results from the Lipschitz-likeness of $Y$ and the inclusion in \eqref{aubin-1}. The above arguments ensures that $x^{*}\in  D^{*}\Sigma\left( \bar{x},\bar{z}\right) \left( z^{*}\right)$. Combining this with \eqref{aubin-1}, the desired result follows.
\end{proof}
Note that our formula in \eqref{DEBobo} is the same as the one obtained in \cite{mordu1}. However, in the later reference, it is required that $z^*$ be in the interior of the corresponding cone; such a requirement is very restrictive and will not make it possible to construct the necessary optimality conditions, which  represent our main goal in this section. Furthermore, under the strong domination property \eqref{dompro}, the paper \cite{liao} provides an estimate of the coderivative of $\varPhi^{E}$ for all $z^{\ast}$ in the uniformly positive polar to cone $\mathbb{R}_{+}^{q}$ defined by
\begin{equation*}
	 	K^{\ast}_{up} := \left\{\alpha \in \mathbb{R}^{q}: \; \exists \beta > 0, \; \langle \alpha,\, z\rangle \geq \beta \| z\|, \; \forall z\in \mathbb{R}_{+}^{q}\right\}.
\end{equation*}
 As it can be seen in Theorem \ref{estime-co-F}, our estimate of the coderivative of $\varPhi^{E}$ is calculated at any point $z^*\in \mathbb{R}^q$, thus enabling an easy derivation of optimality conditions for problem \eqref{optimistic}--\eqref{llpill}, as it will be clear by the end of this section. It is also important to note that a version of the strong domination property can well be defined for $\varPhi^W$.
However, it is unclear how it would help in obtaining an estimate of the coderivative of $\varPhi^W$ analogous to the one derived in Theorem \ref{estime-co-F} for efficient Pareto  points. 

The next proposition gives a sufficient condition for the family of parametric linear programming problems  $\left( \ref{linbil}\right)$ to  satisfy the strong domination property \eqref{dompro}.
\begin{prop}\label{StrongDominationScenario}
Assume that for all $x\in X$,  the set $Y\left( x\right) =\{y\in\mathbb{R}^{m}: \;\, Ax+By\leq d \}$ is bounded. Then, problem \eqref{linbil} satisfies the strong domination property \eqref{dompro}.
\end{prop}
		\begin{proof}
			Fix $x\in X$ and let $\bar{y}\in Y\left( x\right)$. If $\bar{y}\in S\left( x\right)$, then since $C\bar{y}\in C\bar{y}+\mathbb{R}_{+}^{q}$,  we have the inclusion $f\left( x,Y\left( x\right) \right) \subset \varPhi^{E}\left( x\right)  +\mathbb{R}_{+}^{q}$. Suppose that $\bar{y}\notin S\left( x\right)$ and consider the set
			\begin{equation*}
			{\mathcal DP}\left( x,\bar{y}\right) =\left\{y\in G\left( x\right)\left|\;\, C\bar{y}-Cy\in \mathbb{R}_{+}^{q} \right.\right\}.
			\end{equation*}
			Since $Y\left( x\right) $ is bounded, ${\mathcal DP}\left( x,\bar{y}\right)$ is also bounded. Hence, its support function $\sigma\left( \cdot ,{\mathcal DP}\left( x,\bar{y}\right)\right)$ is defined everywhere. Now, choose $u\in \text{int} \ \mathbb{R}_{-}^{q}$. Then, from \cite[Corollary 23.5.3]{rock} there exists $z\in {\mathcal DP}\left( x,\bar{y}\right)$ such that $\langle C^{T}u, z \rangle = \sigma\left( C^{T}u ,{\mathcal DP}\left( x,\bar{y}\right)\right)$. We claim that $z\in S\left( x\right)$. Indeed, suppose, in contrary to our claim, that there exist $v\in Y\left( x\right)$ such that
			\begin{equation*}
				Cv-Cz \in - \mathbb{R}_{+}^{q} \ \ \text{and} \ \ Cv\neq Cz.
			\end{equation*}
			Or equivalently, that
			\begin{equation*}
				Cv-Cz \in - \mathbb{R}_{+}^{q}\setminus \{0\} \ \  \text{for some} \ v\in Y\left( x\right).
			\end{equation*}
			On the one side, $C\bar{y}-Cv=C\bar{y}-Cz+Cz-Cv \in \mathbb{R}_{+}^{q}+\mathbb{R}_{+}^{q}\setminus \{0\}\subset \mathbb{R}_{+}^{q}$. Consequently, $v\in {\mathcal DP}\left( x,\bar{y}\right)$.
			On the other side, since, $u\in \text{int} \ \mathbb{R}_{-}^{q}$, one has $0 < \langle u , Cv-Cz\rangle$. Thus, $\sigma\left( C^{T}u ,{\mathcal DP}\left( x,\bar{y}\right)\right)<  \langle C^{T}u , v\rangle$. Which is a contradiction. Finally, we have $z\in {\mathcal DP}\left( x,\bar{y}\right)$ and $Cz\in \varPhi^{E}\left( x\right)$, it follows that $C\bar{y}-Cz \in\mathbb{R}_{+}^{q}$, which concludes the proof.
		\end{proof}		
		
Now, we come to the final step before the statement of the main result of this section; i.e., we provide an estimate for the coderivative of the lower-level optimal solution set-valued mapping $S$ \eqref{intro2} under the GVFCQ  \eqref{calm1}.
\begin{prop}\label{estime-co-s}
Consider the lower-level optimal solution set-valued mapping $S$ \eqref{intro2} and suppose that $f$ is locally Lipschitz continuous and the set $\mbox{gph} \ Y$ and $\mbox{gph} \ \varPhi$ are closed. Furthermore, assume that the GVFCQ holds at $(\bar x, \bar y)$. Then it holds that
\[
D^*S(\bar x, \bar y)(y^*)  \subset \underset{(u^*, v^*): \;\, u^* \in D^*\varPhi\left(\bar x, \, f(\bar x, \bar y)\right)(-v^*)}{\bigcup}\;\;\underset{(a^*, b^*)\in D^*f(\bar x, \bar y)(v^*)}{\bigcup}
\left\{u^* + a^* + D^*Y(\bar x,\, \bar y)(y^* + b^*)\right\}.
\]
If additionally, $f$ is strictly differentiable, then we have
\[
D^*S(\bar x, \bar y)(y^*)  \subset \underset{(u^*, v^*): \;\, u^* \in D^*\varPhi\left(\bar x,\,f(\bar x, \bar y)\right)(-v^*)}{\bigcup}
\left\{u^* + \nabla_x f(\bar x, \bar y)^\top v^* + D^*Y(\bar x,\, \bar y)\left(y^* + \nabla_y f(\bar x, \bar y)^\top v^*\right)\right\}.
\]
\end{prop}
\begin{proof}
Note that the graph of $S$ \eqref{intro2} can be written as
\[
\begin{array}{l}
\mbox{gph} \ S = \Omega \cap \psi^{-1}\left(\Lambda\right) \;\, \mbox{ with } \;\, \Omega:=\mbox{gph} \ Y,
\;\,
\psi(x,y):=
\left(\begin{array}{c}
      x\\
      f(x,y)
\end{array}\right),\;\,  \mbox{ and }\;\, \Lambda:=\mbox{gph} \ \varPhi.
\end{array}
\]
Then, based on the assumptions made, it follows from \cite[Theorem 4.1]{henrion} that
\[
N\left(\left(\bar x, \bar y\right);\; \mbox{gph}~S\right) \subset
\underset{(u^*, v^*) \in N\left(\psi(\bar x, \bar y); \; \Lambda\right)}{\bigcup}D^*\psi(\bar x, \bar y)(u^*, v^*) \;\; + N\left((\bar x, \bar y); \; \Omega\right).
\]
Hence, considering the definition of the concept of coderivative in \eqref{codDefinition}, we have
\[
\begin{array}{lll}
D^*S(\bar x, \bar y)(y^*)
& \subset & \underset{(u^*, v^*) \in N\left(\psi(\bar x, \bar y); \; \Lambda\right)}{\bigcup}\left\{x^*\in \mathbb{R}^n\left|\; (x^*, -y^*)\in
 N\left((\bar x, \bar y); \; \Omega\right)
+ D^*\psi(\bar x, \bar y)(u^*, v^*) \right.\right\}\\[3ex]
& \subset & \underset{(u^*, v^*) \in N\left(\psi(\bar x, \bar y); \; \Lambda\right)}{\bigcup}\left\{x^*\in \mathbb{R}^n\left|\; (x^*-u^*, -y^*)\in  N\left((\bar x, \bar y); \; \Omega\right)
+ D^*f(\bar x, \bar y)(v^*) \right.\right\}\\[3ex]
& \subset & \underset{(u^*, v^*) \in N\left(\psi(\bar x, \bar y); \; \Lambda\right)}{\bigcup}\left\{x^*\in \mathbb{R}^n\left|\;\exists (a^*, b^*)\in D^*f(\bar x, \bar y)(v^*):\right. \right.\\[3ex]
&         &\qquad \qquad \qquad\qquad\qquad \qquad  \left.\left. x^* - u^* - a^* \in  D^*Y(\bar x,\, \bar y)(y^* + b^*) \right.\right\}\\[2ex]
\end{array}
\]
with the second inclusion resulting from
\[
D^*\psi(\bar x, \bar y)(u^*, v^*) \subset \left(\begin{array}{c}
                                                    u^*\\
                                                     0
                                                     \end{array}\right) + D^*f(\bar x, \bar y)(v^*).
\]
Clearly, the last inclusion in the above sequence of inclusions gives the desired result for the upper bound of $D^*S(\bar x,\,\bar y)(y^*)$ when $f$ is locally Lipschitz continuous. The case where $f$ is strictly differentiable obviously follows from  $D^*f(\bar x, \bar y)(v^*)=\nabla f(\bar x, \bar y)^\top v^*$.
\end{proof}
What is nice with this result is not the construct process of the proof, which is not necessarily new, but its reliance on the GVFCQ and the corresponding rich set of sufficient conditions provided in the previous section. Such an approach does not seem to have been used in the literature to construct an estimate of the coderivative of the optimal solution set-valued mapping of a parametric multiobjective optimization problem.

We are now ready to state one of the main results of this paper, providing new
necessary optimality conditions for the multiobjective bilevel optimization problem \eqref{optimistic}--\eqref{llpill}.
\begin{thm}\label{cnnsooth}
Let $(\bar x, \bar y)$ be a local efficient/weakly efficient Pareto point for problem \eqref{optimistic}--\eqref{llpill}. We assume that the function $F$ and $f$ are Lipschitz continuous around $(\bar x, \bar y)$ and suppose that $X$, $\mbox{gph}~S$, $\mbox{gph}~Y$, and $\mbox{gph}~\varPhi$ are closed sets. Furthermore, assume that the GVFCQ holds at $(\bar x, \bar y)$. Then, there exist vectors $v^{*}\in \mathbb{R}^{q}$ and  $w^{*}\in {\mathbb{R}}_{+}^{p}$ with $\| w^{*}\| =1$ such that
\begin{equation}\label{equ-before}
\begin{array}{l}
0  \in   \partial \langle w^{*},\; F\rangle \left(\bar{x},\bar{y}\right) + \partial \langle v^{*},\; f\rangle \left(\bar{x},\bar{y}\right) +  D^{*}\varPhi\left(\bar{x}, f(\bar x, \bar y)\right)\left( -v^{*}\right)\times \{0\}\\[1ex]
  \qquad \qquad \qquad \qquad \qquad \qquad \qquad + \;\, N\left(\left(\bar{x},\bar{y}\right);\; \mbox{gph}~Y\right)+ N\left(\bar{x};\; X\right)\times \{0\}.
\end{array}
\end{equation}
If $\varPhi = \varPhi^E$ in \eqref{equ-before} and additionally, gph$\Sigma$ is locally compact around $\bar{x}$, $Y$ is locally closed and Lipschitz-like around $\left( \bar{x},\bar{y}\right)$, and the strong domination property \eqref{dompro} is satisfied.  Then, there exist vectors $v^{*}\in \mathbb{R}^{q}$, $\left( \alpha^{*},\beta^{*}\right) \in \partial\langle - v^*, \; f\rangle \left(\bar{x},\bar{y}\right)$, and $w^{*}\in {\mathbb{R}}_{+}^{p}$ with $\| w^{*}\| =1$ such that
	\begin{equation}\label{insert-1}
	\begin{array}{l}
	\left(-\alpha^{*},\, 0\right)   \in   \partial \langle w^{*}, F\rangle \left(\bar{x},\bar{y}\right) + \partial \langle v^*, f\rangle\left(\bar{x},\bar{y}\right) +  D^{*}Y\left(\bar{x},\bar{y}\right)\left(\beta^*\right) \times \{0\}\\
	  \qquad \qquad \qquad \qquad \qquad \qquad \qquad + \;\, N\left( \left(\bar{x},\bar{y}\right);~\text{gph}~Y\right) + N\left(\bar{x};\; X\right)\times \{0\}.
	\end{array}
	\end{equation}
If additionally, $F$ and $f$ are strictly differentiable at the point $(\bar x, \bar y)$, then there exist vectors $v^{*}\in \mathbb{R}^{q}$ and $w^{*}\in {\mathbb{R}}_{+}^{p}$ with $\| w^{*}\| =1$ such that we have
	\begin{equation}\label{thm-cn1fp1}
\begin{array}{lll}
0 \in \nabla_x F\left(\bar{x}, \bar{y}\right)^\top w^{*} & + & D^{*}Y\left(\bar{x},\bar{y}\right)\left(-\nabla_{y}f\left(\bar{x},\bar{y}\right)^\top v^{*}\right)\\
 & + & D^{*}Y\left(\bar{x},\bar{y}\right)\left(\nabla_{y}F\left(\bar{x},\bar{y}\right)^\top w^{*} + \nabla_{y}f\left(\bar{x},\bar{y}\right)^\top v^{*}\right) + N\left(\bar x; \; X\right).
\end{array}
\end{equation}
%
\end{thm}
\begin{proof}
First start by noting that based on \eqref{mapmapmap},  problem \eqref{optimistic}--\eqref{llpill} can be rewriting as
\begin{equation*}
\mathbb{R}^{p}_+ -\min~F(x,y) \;\; \mbox{ s.t. }\;\; (x, y) \in \Pi.
\end{equation*}
Since the set $\Pi$ is closed, as intersection of two closed sets, and the function $F$ is Lipschitz continuous around the point $(\bar x, \bar y)$, which is a local efficient/weakly efficient Pareto point for problem \eqref{optimistic}--\eqref{llpill}, there exists $w^*\in \mathbb{R}^p_+$ with $\|w^*\|=1$ such that we have
 \begin{equation}\label{cn}
   \partial \langle w^*, \; F\rangle (\bar x, \bar y) + N\left((\bar x, \bar y); \;\Pi\right),
 \end{equation}
 according to \cite[Theorem 5.3]{BaoMordukhovich}. Hence, it suffices now to calculate an upper estimate of the normal cone to $\Pi$.
Using the intersection rule from \cite[Corollary 3.5]{mordu3}, one has
\begin{equation}\label{cosolmap}
N\left( \left(\bar{x},\bar{y}\right);\;\Pi\right)  \subset N\left( \left(\bar{x},\bar{y}\right); \; X\times\mathbb{R}^{m}\right) +N\left( \left(\bar{x},\bar{y}\right);\; \text{gph}~S\right)
\end{equation}
as $X$ and $\text{gph}~S$ are locally closed around $\bar{x}$ and $\left(\bar{x}, \bar{y}\right)$, respectively, and provided that  \begin{equation}\label{cq-cq}
N\left(\left(\bar{x},\bar{y}\right);\; X\times \mathbb{R}^{m}\right)  \cap
\left(-N\left(\left(\bar{x},\bar{y}\right);\; \text{gph}~S\right)\right)   =\{0\}.
\end{equation}
is satisfied. Considering the coderivative calculus rules, we can easily show that the fulfilment of \eqref{mqc} implies that \eqref{cq-cq} holds.
Then combining \eqref{cn} and \eqref{cosolmap}, one gets
\begin{equation}\label{sum-cone}
0\ \in \ \partial \langle w^{*}, \; F\rangle \left(\bar{x},\bar{y}\right) + N\left(\left(\bar{x},\bar{y}\right); \;X\times\mathbb{R}^{m}\right) +N\left( \left(\bar{x},\bar{y}\right); \;\text{gph} \ S\right).
\end{equation}
Hence, there exist
$\left( x^{*}, y^{*}\right)\in \partial \langle w^{*}, F\rangle \left(\bar{x},\bar{y}\right)$ and $c^{*}\in N\left(\bar{x}; \;X\right)$ such that
\begin{equation}\label{Patou}
\left( -x^{*}-c^{*},-y^{*}\right) \in N\left( \left(\bar{x},\bar{y}\right);\; \text{gph} \ S\right) \;\, \mbox{ or equivalently, } \;\, -x^{*}-c^{*} \in D^{*}S\left(\bar{x},\bar{y}\right)\left( y^{*}\right).
\end{equation}
Thanks to the upper estimate of coderivative of the optimal  solution set-valued mapping $S$ in  Proposition \ref{estime-co-s}, we can find vectors $u^{*}$, $v^{*}$, $a^{*}$ and $b^{*}$ such that
\begin{equation}\label{Patata}
\left.\begin{array}{r}
 u^{*} \in D^{*}\varPhi\left( \bar{x}, \,f(\bar x, \bar y)\right) \left( -v^{*}\right)   \\
 \left( a^{*},b^{*}\right) \in D^{*}f\left( \bar{x},\bar{y}\right) \left( v^{*}\right) \\
 -x^{*}-c^{*} \in u^{*}+a^{*}+D^{*}Y\left( \bar{x}, \bar{y}\right) \left( y^{*}+b^{*}\right)
\end{array}\right\}
\end{equation}
given that gph$Y$ and gph$\varPhi$ are closed, and $f$ is Lipschitz continuous around $(\bar x, \bar y)$.
Then combining \eqref{sum-cone}, \eqref{Patou}, and \eqref{Patata}, we immediately arrive at \eqref{equ-before}.

As for the inclusion in \eqref{insert-1}, it follows from Theorem \ref{estime-co-F} that one has an upper estimate for coderivative of frontier map $\varPhi^{E}$ with respect to local Pareto optimality concept
\begin{equation}\label{comap1}
D^{*}\varPhi^{E}\left( \bar{x},\bar{z}\right)\left( -v^{*}\right) \subseteq {\displaystyle \bigcup_{\left( \alpha^{*},\beta^{*}\right)\in D^{*}f\left( \bar{x},\bar{y}\right)\left( -v^{*}\right) }} \bigg [ \alpha^{*}+D^{*}Y\left( \bar{x},\bar{y}\right)\left( \beta^{*}\right)\bigg ]
\end{equation}
considering the assumptions made. Substituting $\left( \ref{comap1}\right)$ into  $\left( \ref{equ-before}\right)$, it follows that we can find $\left( \alpha^{*},\beta^{*}\right) \in D^{*}f\left( \bar{x},\bar{y}\right)\left( -v^{*}\right)$ such that we have \eqref{insert-1}, which obviously leads to \eqref{thm-cn1fp1} under the additional differentiability assumptions.
\end{proof}

\begin{rem}
Recall that the CQ \eqref{mqc} is automatically satisfied  at $\left( \bar{x},\bar{y}\right)$ provided that problem \eqref{optimistic}--\eqref{llpill}  has no upper-level constraints (i.e., $X =\mathbb{R}^{n}$) or the lower-level optimal solution set-valued mapping $S$ is Lipschitz-like  around $\left( \bar{x},\bar{y}\right)$, which is automatically the case if an upper estimate of $D^*S(\bar x, \bar y)(0)$ from Proposition \ref{estime-co-s} equals to zero, thanks to the Mordukhovich criterion \cite{mordu2, mordu3}; also see \cite{zemko1, zemko0Est} for further details and references.
\end{rem}

To have a clear view of the fact that the necessary optimality conditions obtained in Theorem \ref{cnnsooth} represent a natural extension of the those from a standard optimistic bilevel optimization problem, consider problem \eqref{optimistic}--\eqref{llpill} with $p=1$ and $q=1$. Let $(\bar x, \bar y)$ be a local optimal solution of the problem in this case. If the point satisfies the corresponding version of CQ \eqref{cq-cq} and $F$ and $f$ are strictly differentiable, then we have
\begin{equation}\label{Fatou}
0\in \nabla_x F(\bar x, \bar y)  + D^*S(\bar x, \bar y)(\nabla_y F(\bar x, \bar y)) + N(\bar x; \; X).
\end{equation}
This inclusion obviously coincides with \eqref{sum-cone} in this context where $w^*$ reduces to $1$. Secondly, if $\varphi$ denotes the optimal value function of the corresponding parametric optimization problem  \eqref{llpill} and we additionally suppose that the function $\varphi$ is lower semicontinuous around $\bar x$ and the set-valued mapping
\begin{equation}\label{CalPhi2New}
\Psi_{\varphi}(v):=\left\{\left( x,y\right) \in \text{gph} \ Y: \;\; f(x,y)-\varphi(x) + v =0\right\}.
\end{equation}
is calm at $(0, \bar x, \bar y)$, then condition \eqref{Fatou} can be detailed further to obtain
\begin{equation}\label{Fatouna}
0\in \nabla F(\bar x, \bar y) + \nabla f(\bar x, \bar y)v^* + \partial \langle -v^*, \; \varphi \rangle (\bar x) \times \{0\} + N((\bar x, \bar y); \; \mbox{gph}~Y) + N(\bar x; \; X)\times \{0\}
\end{equation}
for some $v^*\in \mathbb{R}$. Similarly, this coincides with \eqref{equ-before} for $w^*=1$. Note that the set-valued mapping \eqref{CalPhi2New} is slightly different from \eqref{CalPhi2}, as in the latter case, we instead have an inequality on the perturbed value function constraint. Of course, using the version of the set-valued mapping in \eqref{CalPhi2} would have led to $v^*\geq 0$.

Finally, still in the case  $p=1$ and $q=1$, if $S$ is inner semicontinuous and has a closed graph around $(\bar x, \bar y)$, and $\varphi$ is Lipschitz continuous around $\bar x$, then from \eqref{Fatouna}, we have inclusion \eqref{thm-cn1fp1} with the corresponding $w^*=1$. For more background details on the constructions and relevant concepts above, in the context of standard optimistic optimization, interested readers are referred to \cite{zemko1}, where, unlike in \eqref{thm-cn1fp1}, a scalarization approach is used to deal with the lower-level multiobjective problem, as, to the best of our knowledge, it is the case for all previous references on necessary optimality conditions for multiobjective bilevel optimization. 

\section{Application to smooth constraint functionals}
Let us consider the multiobjective bilevel optimization problem \eqref{optimistic}--\eqref{llpill} in the case where the upper- and lower-level feasible sets are defined by
\[
X:=\left\{x\in \mathbb{R}^n:\; G(x)\leq 0\right\} \; \mbox{ and }\; Y(x):=\left\{y\in \mathbb{R}^m:\; g(x, y)\leq 0\right\},
\]
respectively, with $G :\mathbb{R}^n \longrightarrow \mathbb{R}^r$ and $g :\mathbb{R}^n\times \mathbb{R}^m \longrightarrow \mathbb{R}^s$ being continuously differentiable functions. The upper-level regularity condition will be said to hold at $\bar x$ if there exists a vector $d\in \mathbb{R}^n$ such that we have
\[
\nabla G_i(\bar x)^\top d < 0 \;\, \mbox{ for all }\;\, i\in I_G(\bar x):=\left\{i\in \{1, \ldots, r\}:\;\, G_i(\bar x)=0\right\}.
\]
Similarly, the lower-level regularity condition will be satisfied at $(\bar x, \bar y)$ if
if there exists a vector $d\in \mathbb{R}^{n+m}$  that verifies
\[
\nabla_y g_j(\bar x, \bar y)^\top d < 0 \;\, \mbox{ for all }\;\, j\in I_g(\bar x, \bar y):=\left\{j\in \{1, \ldots, s\}:\;\, g_j(\bar x, \bar y)=0\right\}.
\]
Obviously, these upper- and lower-level regularity conditions correspond to the Mangasarian-Fromovitz constraint qualification for the feasible set of the corresponding (upper- or lower-) level of our problem \eqref{optimistic}--\eqref{llpill}. It is well-known that under the upper- and lower-level regularity conditions, we respectively have
\[
N(\bar x; \; X) = \left\{\nabla G(\bar x)^\top u: \; u\geq 0, \; u^\top G(\bar x)=0\right\}
\]
and
\[
D^*Y(\bar x, \bar y)(y^*) = \left\{\nabla_x g(\bar x, \bar y)^\top v:\; -y^*= \nabla_y g(\bar x, \bar y)^\top v, \; v\geq 0,\; v^\top g(\bar x, \bar y)=0\right\}.
\]

Now, in addition to all the assumptions of Theorem \ref{cnnsooth}, we assume that the upper- and lower-level regularity conditions at $\bar x$ and $(\bar x, \bar y)$, respectively, it follows from \eqref{thm-cn1fp1} that there exist $w^*\in \mathbb{R}^p_+$ with $\|w^*\|=1$, $v^*\in \mathbb{R}^q$, $u\in \mathbb{R}^r$, and $v\in \mathbb{R}^s$ satisfying the relationships
\begin{eqnarray}
\nabla_y f(\bar x, \bar y)^\top (-v^*) +  \nabla_y g(\bar x, \bar y)^\top v=0,\label{Uka1}\\
u\geq 0, \;\; G(\bar x)\leq 0, \;\; u^\top G(\bar x)=0, \label{Uka2}\\
v\geq 0, \;\; g(\bar x, \bar y)\leq 0, \;\; v^\top g(\bar x, \bar y)=0,\label{Uka3}
\end{eqnarray}
and such that
\[
- \nabla_x F\left(\bar{x}, \bar{y}\right)^* w^{*} - \nabla G(\bar x)^\top u
 - \nabla_x g(\bar x, \bar y)^\top v \in D^{*}Y\left(\bar{x},\bar{y}\right)\left(\nabla_{y}F\left(\bar{x},\bar{y}\right)^{*}w^{*} +  \nabla_{y}f\left(\bar{x},\bar{y}\right)^{*}v^{*}\right).
\]
Then from a second application of the above coderivative formula for $D^*Y(\bar x, \bar y)(y^*)$ to the latter inclusion, it follows that we can find $w\in \mathbb{R}^s$, $w^*\in \mathbb{R}^p_+$ with $\|w^*\|=1$, $u\in \mathbb{R}^r$, and $v\in \mathbb{R}^s$ such that the relationships \eqref{Uka1}--\eqref{Uka3} hold together with
\begin{eqnarray}
\nabla F(\bar x, \bar y)^\top w^* + \left[\begin{array}{c}
                                            \nabla G(\bar x)^\top u\\
                                            0
                                          \end{array}\right]
 + \nabla g(\bar x, \bar y)^\top (v + w) =0,\label{Uka4}\\
w\geq 0, \;\; g(\bar x, \bar y)\leq 0, \;\; w^\top g(\bar x, \bar y)=0.\label{Uka5}
\end{eqnarray}
The optimality conditions \eqref{Uka1}--\eqref{Uka5} are very similar to their standard optimistic bilevel optimization problem counterpart with scalar objective functions, as it can be seen in \cite{DDM2007,DZ2011}, for example. The corresponding conditions in the latter paper have been shown in the recent papers \cite{FischerZemkohoZhou,FliegeTinZemkoho,ZemkohoZhouTheoretical} to be suitable to efficiently solve standard optimistic bilevel optimization problem. Hence, the extension of the methods in these papers to multiobjective bilevel programs will be explored in future works.

Finally, to end this section, we provide the following illustrative example, where the lower-level problem is the linear parametric multiobjective problem from Example \ref{FirstExample}.
\begin{exmp}
Consider problem \eqref{optimistic}--\eqref{llpill}, where $F :\mathbb{R}^2\times \mathbb{R}^2 \rightarrow \mathbb{R}^p$ is any differentiable function and  $X$ and the lower-level problem are defined as in Example \ref{FirstExample}. We can easily check that for all $x\in X$, $Y(x):=\left\{y\in\mathbb{R}^2|\; Ax + By\leq d \right\}$ is bounded. Hence, the strong domination property \eqref{dompro} is satisfied according to Proposition \ref{StrongDominationScenario}. Furthermore, all the other assumptions of Theorem \ref{cnnsooth} hold. Hence, for any local efficient Pareto point $(\bar x, \bar y)$ of the problem,
\begin{eqnarray*}
\nabla_x F(\bar x, \bar y)^\top w^* - \left(\begin{array}{c}
                                              u_1\\
                                              u_2
                                            \end{array}\right) - \left(\begin{array}{c}
                                              v_5 + w_5\\
                                              v_6 + w_6
                                            \end{array}\right)=0,\\
\nabla_y F(\bar x, \bar y)^\top w^* + \left(\begin{array}{r}
                                              v_1+w_1 - v_2-w_2 + v_5+w_5\\
                                              2(v_3 + w_3) -v_4-w_4 + v_6 +w_6
                                            \end{array}\right)=0,\\
\left(\begin{array}{r}
                                              2v^*_1\\
                                              v^*_2
                                            \end{array}\right) - \left(\begin{array}{r}
                                              v_1 - v_2 + v_5\\
                                              2v_3 -v_4 + v_6
                                            \end{array}\right)=0,\\
u_1\geq 0, \;\; y_1\geq 4, \;\; u_1 (y_1 - 4)=0,\\
u_2\geq 0, \;\; y_2\geq 3, \;\; u_2 (y_2 - 3)=0,\\
v\geq 0, \;\; A\bar x + B\bar y \leq d, \;\; v\left(A\bar x + B\bar y -d\right)=0,\\
w\geq 0, \;\; A\bar x + B\bar y \leq d, \;\; w\left(A\bar x + B\bar y -d\right)=0,
\end{eqnarray*}
for some $u\in \mathbb{R}^2$, $v\in \mathbb{R}^6$, $w\in \mathbb{R}^6$, $v^{*}\in \mathbb{R}^{2}$, and  $w^{*}\in {\mathbb{R}}_{+}^{p}$ with $\| w^{*}\| =1$. Note that here, the Pareto efficient solution concept is also considered for the lower-level problem. The matrices $A$, $B$, and $d$ in the last two lines of this system are given in  \eqref{MaterialExample}.
\end{exmp}
\section*{Acknowledgements}  The work of AZ is supported by the EPSRC grant EP/V049038/1 and the Alan Turing Institute for Data Science and Artificial Intelligence under the EPSRC grant EP/N510129/1

\section*{Data availability statement}
No data is needed in this paper.

\end{document}